\documentclass[review]{elsarticle}

\usepackage{lineno,hyperref}
\usepackage{amssymb,enumerate}
\usepackage{amsmath}
\usepackage{graphics}
\numberwithin{equation}{section}

\def\R{{\bf R}}

\def\d{\displaystyle}
\def\e{{\varepsilon}}

\def\wt{\widetilde}

\newtheorem{thm}{Theorem}[section]

\newtheorem{lem}{Lemma}[section]

\newtheorem{rem}{Remark}[section]

\newtheorem{proof}{Proof}[section]

\begin{document}

\begin{frontmatter}

\title{Weighted $L^2-L^2$ estimate for wave equation and its applications}

\author[mymainaddress,secondaryaddress]{Ning-An Lai\corref{mycorrespondingauthor}}
\cortext[mycorrespondingauthor]{Corresponding Author}
\ead{hyayue@gmail.com}

\address[mymainaddress]{Institute of Nonlinear Analysis and Department of Mathematics,\\ Lishui University, Lishui 323000, China}
\address[secondaryaddress]{School of Mathematical Sciences, Fudan University, Shanghai 200433, China}

\begin{abstract}

In this work we establish a weighted $L^2-L^2$ estimate for inhomogeneous wave equation in 3-D, by introducing a Morawetz multiplier with weight of power $s(1<s<2)$,
and then integrating on the light cones and $t$ slice. With this weighted $L^2-L^2$ estimate in hand, we may give a new proof of global existence
for small data Cauchy problem of semilinear wave equation with supercritical power in 3-D. What is more, by combining the Huygens' principle for wave equations in 3-D, the global existence for semilinear wave equation with scale invariant damping in 3-D is established.

\end{abstract}


\end{frontmatter}


\section{Introduction}
\par\quad
In this paper, we are devoted to establishing a weighted $L^2-L^2$ estimate for the inhomogeneous wave equation with zero initial data in 3-D
\begin{equation}
\label{equationnodata}
\partial_t^2\phi-\Delta \phi=F(t, x),~~(t, x)\in \R_+\times \R^3.\\
\end{equation}
For this purpose, we introduce a Morawetz multiplier, which is showed in \eqref{multiplier} below, with a power $s$ in the weight varying from 1 to 2. After making computation in the polar coordinates, we integrate the resulted equality on the time slice and the light cones. Finally, by combining Hardy type inequalities, we establish the weighted $L^2-L^2$ estimate for the solution
of \eqref{equationnodata}.

As one application, we give a new proof of global existence for small data Cauchy problem of semilinear wave equation
\begin{equation}
\label{semiequation}
\partial_t^2\phi-\Delta \phi=|\phi|^p,~~(t, x)\in \R_+\times \R^3,\\
\end{equation}
for $1+\sqrt 2<p<3$.
\begin{rem}
In \cite{LZ}, the author and Zhou used the multiplier \eqref{multiplier} with $s=1$ to establish a similar weighted $L^2-L^2$
estimate for the inhomogeneous wave equation, and furthermore demonstrated an elementary proof of global existence for small data Cauchy proof \eqref{semiequation} in high dimensions$(n\geq 4)$.
\end{rem}
There have been a long history for the study of the large time behavior of the small data Cauchy problem of the following semilinear equation in all dimensions
\begin{equation}
\label{ndsemiequation}
\partial_t^2\phi-\Delta \phi=|\phi|^p,~~(t, x)\in \R_+\times \R^n.\\
\end{equation}
The pioneering work is due to John \cite{John}, in which he proved global existence for the Cauchy problem \eqref{semiequation}
with smooth data for $p>1+\sqrt 2$ and blow-up for $1<p<\sqrt 2$ with appropriate assumption on the data. Then Strauss \cite{Strauss} announced a critical power $p_c(n)(n\geq 2)$ for the small data Cauchy problem of \eqref{ndsemiequation}, which is the positive root of the following quadratic equation
\[
(n-1)p^2-(n+1)p-2=0.
\]
Here critical power means that if $p\in (1, p_c(n)]$ there will be no global solutions while if $p\in (p_c(n), \infty)$ the solutions exist globally in time. This conjecture has been completely solved recently, we refer the reader to
\cite{Glassey1, Glassey2, Sideris, Rammaha, Schaeffer, Zhou1, Georgive, Takamura, Yordanov, Zhou2} and references therein.
The key point for global existence in \cite{John}, John established the following weighted pointwise estimate for the solution to \eqref{equationnodata} with zero data
\begin{equation}
\begin{aligned}\label{johnpointwise}
\left\|t(t-r)^{p-2}\phi\right\|_{L^{\infty}(\R_+^{1+3})}\leq C_p\left\|t^p(t-r)^{p(p-2)}F\right\|_{L^{\infty}(\R_+^{1+3})},
\end{aligned}
\end{equation}
for $1+\sqrt 2<p\leq 3$, where $r=|x|$ and $C_p$ denotes a constant depending on $p$. Lindblad and Sogge \cite{LS}
established an analogy of \eqref{johnpointwise} without the weight $|t-r|$, which involves the $L^p$ norm of the radial variables,
and showed a somewhat simpler proof. We point out that both the methods in \cite{John} and \cite{LS} heavily rely on the explicit expression of the solution of the linear wave equation in 3-D. Not much later, Georgiev, Lindblad and Sogge \cite{Georgive} established a weighted Strichartz estimate with simpler weights which are invariant under Lorentz rotations, and can give an elementary proof of global existence for small data Cauchy problem \eqref{semiequation} with $p> 1+\sqrt 2$. In this work, based on the weighted $L^2-L^2$ estimate, we are going to establish another analogy of \eqref{johnpointwise}, which includes the weights $|t+r|$ and $|t-r|$, and $L^2$ normal of the angular variable. Then we may show global existence for the small data Cauchy problem \eqref{semiequation} for $1+\sqrt 2<p<3$, avoiding to use the explicit formula of the solution.

Another application of our weighted $L^2-L^2$ estimate is to prove the global existence for small data Cauchy problem of semilinear wave equation with scale invariant damping in 3-D
\begin{equation}
\begin{aligned}\label{damped}
\partial_t^2\Phi-\Delta \Phi+\frac{2}{1+t}\Phi_t=|\Phi|^p,~~(t, x)\in \R_+\times \R^3.\\
\end{aligned}
\end{equation}
Here, scale invariant means that, under the following scaling
\[
\wt{\Phi}(t, x):=\Phi(\sigma(1+t)-1, \sigma x),\ \sigma>0,
\]
the equation
\begin{equation}
\begin{aligned}\label{nddamped}
\partial_t^2\Phi-\Delta \Phi+\frac{2}{1+t}\Phi_t=|\Phi|^p,~~(t, x)\in \R_+\times \R^n\\
\end{aligned}
\end{equation}
remains invariant. By using the the so-called Liouville transformation
\begin{equation}\nonumber
\begin{aligned}
\phi(t, x):=(1+t)\Phi(t, x),
\end{aligned}
\end{equation}
equation \eqref{nddamped} will be changed into
\begin{equation}
\begin{aligned}\label{ndnodamped}
\partial_t^2\phi-\Delta \phi=\frac{|\phi|^p}{(1+t)^{p-1}},~~(t, x)\in \R_+\times \R^n.\\
\end{aligned}
\end{equation}
For the small data Cauchy problem \eqref{nddamped}, or equivalently, \eqref{ndnodamped}, we know it admits a critical power
$p=p_c(n+2)$, due to the known results. We refer the reader to \cite{wak16, DLR15, D-L, Pa} and references therein. However, the global existence results in high dimensions$(p>p_c(n+2), n\geq 3)$ have only been established under symmetrical assumption. See \cite{DLR15} for $n=3$, \cite{D-L} for odd and higher dimensions$(n\geq 5)$ and \cite{Pa} for $n\geq 4$. We aim to prove global existence to the small data Cauchy problem \eqref{damped} for $p>p_c(5)$ without radial symmetrical assumption. However, compared to the Cauchy problem of \eqref{semiequation}, we have two difficulties. Firstly, we may take derivative to the equation \eqref{damped} only one time at most, due to the fact $1<p_c(5)<2$. Secondly, an additional weight will appear in the nonlinear term after making the Liouville transformation. We overcome the first difficulty by establishing a new weighted estimate to the solution, based on Sobolev embedding and our weighted $L^2-L^2$ estimate, while the second one is overcame by using the Strong Huygens' principle for wave equations in 3-D. We demonstrate the details in section 5.

\begin{rem}
For the small data Cauchy problem of
\begin{equation}
\begin{aligned}\label{nddamped1}
\partial_t^2\Phi-\Delta \Phi+\frac{\mu}{1+t}\Phi_t=|\Phi|^p,~~(t, x)\in \R_+\times \R^n,\\
\end{aligned}
\end{equation}
it is conjectured that the critical power is $p=p_c(n+\mu)$. There are some partial results on blow-up and lifespan estimate,
see \cite{IKTW, KS, KTW, LTW, IS, Tu} and references therein.
\end{rem}

\section{Weighted $L^2-L^2$ estimate}
\par
In this section we first establish the weighted $L^2-L^2$ estimate for the inhomogeneous wave equation with zero initial data.
\begin{thm}
\label{weighted}
Consider the following Cauchy problem of wave equation
\begin{equation}
\label{linearequation}
\left\{
\begin{array}{l}
\d \partial_{t}^2\phi-\Delta \phi=F(t, x),
\quad (t, x)~ \mbox{in}\ ~\R_{+}\times \R^3,\\
\phi(0, x)=\phi_t(0, x)=0, \quad x\in\R^3.
\end{array}
\right.
\end{equation}
Suppose that
\[
supp~ F\subset \{|x|\le t+1\},
\]
then we have
\begin{equation}
\begin{aligned}\label{Weightedinequality}
&\sup_{t}\Big(\left\|(t+2-r)^{s/2}\nabla_{t, x}\phi\right\|_{L^2(\R^3)}+\left\|(t+2-r)^{s/2}\left(\frac{\phi}{r}\right)\right\|_{L^2(\R^3)}\\
&+\left\|(t+2-r)^{s/2}\left(\frac{\nabla_{S^2}\phi}{r}\right)\right\|_{L^2(\R^3)}\Big)\\
\leq& C_{\delta}\left\|(t+2+r)^{s/2}(t+r-2)^{1/2+\delta}F\right\|_{L^2(\R_{+}\times \R^3)},
\end{aligned}
\end{equation}
for any constant $\delta>0$ and $1<s<2$. And $C_{\delta}$ denotes a positive constant depending on $\delta$.
\end{thm}
\begin{proof}
We introduce the following multiplier
\begin{equation}
\begin{aligned}\label{multiplier}
X=&(t+2+r)^{s}(\partial_t+\partial_r)+(t+2-r)^{s}(\partial_t-\partial_r)\\
&+\frac1r\left((t+2+r)^s-(t+2-r)^s\right)\\
=&(t+2+r)^{s}\left(\partial_t+\partial_r+\frac1r\right)+(t+2-r)^{s}\left(\partial_t-\partial_r-\frac1r\right).
\end{aligned}
\end{equation}
Multiplying the equation in problem \eqref{linearequation} with the above multiplier, we have
\[
X\phi(\partial_t^2\phi-\Delta \phi)r^2=X\phi Fr^2.
\]

Next we shall compute the left side of the above equality in polar coordinate in several steps. First, we have
\begin{equation}
\begin{aligned}\label{firstpart}
&(t+2+r)^{s}(\phi_t+\phi_r)r^2\Box\phi\\
=&(\partial_t-\partial_r)\left[\frac12(t+2+r)^sr^2(\phi_t+\phi_r)^2\right]+r(t+2+r)^s(\phi_t^2-\phi_r^2)\\
&+(\partial_t+\partial_r)\left[\frac12(t+2+r)^s(\nabla_{S^2}\phi)^2\right]-s(t+2+r)^{s-1}(\nabla_{S^2}\phi)^2\\
&-\nabla_{S^2}\left[(t+2+r)^s(\phi_t+\phi_r)\nabla_{S^2}\phi\right],
\end{aligned}
\end{equation}
where
\[
(\nabla_{S^2}\phi)^2=\left(\partial_{\theta}\phi\right)^2+\frac{1}{\sin^2\theta}(\partial_{\varphi}\phi)^2
\]
if we use the following polar coordinates
\begin{equation}
\nonumber
\left\{
\begin{array}{l}
\d x=r\sin\theta\cos\varphi,\\
y=r\sin\theta\sin\varphi,\\
z=r\cos\theta.
\end{array}
\right.
\end{equation}
In a similar way, we may compute the second part of the multiplier as
\begin{equation}
\begin{aligned}\label{secondpart}
&(t+2-r)^{s}(\phi_t-\phi_r)r^2\Box\phi\\
=&(\partial_t+\partial_r)\left[\frac12(t+2-r)^sr^2(\phi_t-\phi_r)^2\right]-r(t+2-r)^s(\phi_t^2-\phi_r^2)\\
&+(\partial_t-\partial_r)\left[\frac12(t+2-r)^s(\nabla_{S^2}\phi)^2\right]-s(t+2-r)^{s-1}(\nabla_{S^2}\phi)^2\\
&-\nabla_{S^2}\left[(t+2-r)^s(\phi_t-\phi_r)\nabla_{S^2}\phi\right].
\end{aligned}
\end{equation}
For the third part of the multiplier we have
\begin{equation}
\begin{aligned}\label{thirdpart}
&\left[(t+2+r)^s-(t+2-r)^s\right]\phi r\Box \phi\\
=&\left[(t+2+r)^s-(t+2-r)^s\right]r\left\{\Box\frac{\phi^2}{2}-\left[\phi_t^2-(\nabla_x\phi)^2\right]\right\}\\
=&(\partial_t-\partial_r)\left[(t+2+r)^s(\partial_t+\partial_r)\left(\frac{r\phi^2}{2}\right)\right]\\
&-(\partial_t+\partial_r)
\left[(t+2-r)^s(\partial_t-\partial_r)\left(\frac{r\phi^2}{2}\right)\right]\\
&-\left[(t+2+r)^s-(t+2-r)^s\right]r(\phi_t^2-\phi_r^2)\\
&+\frac1r\left[(t+2+r)^s-(t+2-r)^s\right](\nabla_{S^2}\phi)^2\\
&-\nabla_{S^2}\left\{\frac1r\left[(t+2+r)^s-(t+2-r)^s\right]\nabla_{S^2}\left(\frac{\phi^2}{2}\right)\right\}.
\end{aligned}
\end{equation}
Adding the three parts, i.e. \eqref{firstpart}-\eqref{thirdpart}, together, we have
\begin{equation}
\begin{aligned}\label{threeparts}
&(\partial_t-\partial_r)\Big\{\frac12(t+2+r)^sr^2(\phi_t+\phi_r)^2\\
+&(t+2+r)^s\frac{\phi^2}{2}+(t+2+r)^sr\phi(\phi_t+\phi_r)
+\frac12(t+2-r)^s(\nabla_{S^2}\phi)^2\Big\}\\
+&(\partial_t+\partial_r)\Big\{\frac12(t+2-r)^sr^2(\phi_t-\phi_r)^2\\
+&(t+2-r)^s\frac{\phi^2}{2}-(t+2-r)^sr\phi(\phi_t-\phi_r)
+\frac12(t+2+r)^s(\nabla_{S^2}\phi)^2\Big\}\\
+&\frac1r\left[(t+2+r)^s-(t+2-r)^s-sr(t+2+r)^{s-1}-sr(t+2-r)^{s-1}\right]\\
\times&(\nabla_{S^2}\phi)^2-\nabla_{S^2}\left(X\phi\nabla_{S^2}\phi\right)\\
=&(\partial_t-\partial_r)\left[\frac{(t+2+r)^s}{2}r^2\left(\phi_t+\phi_r+\frac{\phi}{r}\right)^2+\frac
{(t+2-r)^s}{2}(\nabla_{S^2}\phi)^2\right]\\
+&(\partial_t+\partial_r)\left[\frac{(t+2-r)^s}{2}r^2\left(\phi_t-\phi_r-\frac{\phi}{r}\right)^2+\frac
{(t+2+r)^s}{2}(\nabla_{S^2}\phi)^2\right]\\
+&\frac1r\left[(t+2-(s-1)r)(t+2+r)^{s-1}-(t+2-r)^{s-1}(t+2+(s-1)r)\right]\\
\times&(\nabla_{S^2}\phi)^2-\nabla_{S^2}\left(X\phi\nabla_{S^2}\phi\right)\\
=&X\phi Fr^2.\\
\end{aligned}
\end{equation}
It is easy to see that
\[
\begin{aligned}
&(t+2-(s-1)r)(t+2+r)^{s-1}\\
=&(t+2)^s\left(1-(s-1)\frac{r}{t+2}\right)\left(1+\frac{r}{t+2}\right)^{s-1}.
\end{aligned}
\]
By Taylor expansion we may verify that for $1<s<2$ and $r\leq t+1$
\[
\begin{aligned}
&\left(1-(s-1)\frac{r}{t+2}\right)\left(1+\frac{r}{t+2}\right)^{s-1}\\
\geq& \left(1-\frac{r}{t+2}\right)^{s-1}\left(1+(s-1)\frac{r}{t+2}\right),
\end{aligned}
\]
and hence
\begin{equation}
\label{taylor}
\left(t+2-(s-1)r\right)\left(t+2+r\right)^{s-1}\geq \left(t+2-r\right)^{s-1}\left(t+2+(s-1)r\right).
\end{equation}
Integrating \eqref{threeparts} on the time slice $t=constant$, the lightcone $u=t-r=constant$ and $\underline{u}=t+r=constant$, respectively, and noting the inequality \eqref{taylor}, we have
\begin{equation}
\begin{aligned}\label{integrating}
&\sup_{t}\int_{\R^3}\Big\{(t+2+r)^sr^2\left(\phi_t+\phi_r+\frac{\phi}{r}\right)^2+(t+2-r)^s\left(\phi_t-\phi_r-\frac{\phi}{r}\right)^2\\
&+\left[(t+2+r)^s+(t+2-r)^s\right]\frac{(\nabla_{S^2}\phi)^2}{r^2}\Big\}dx\\
&+\sup_u\int_{t-r=u}\int_{S^2}(t+2+r)^sr^2\left(\phi_t+\phi_r+\frac{\phi}{r}\right)^2d\omega dr\\
&+\sup_{\underline{u}}\int_{t+r=\underline{u}}\int_{S^2}(t+2-r)^sr^2\left(\phi_t-\phi_r-\frac{\phi}{r}\right)^2d\omega dr\\
&\leq\int\int\int|X\phi||F|r^2drd\omega dt.\\
\end{aligned}
\end{equation}
We are going to estimate the right hand side. Recalling the definition of the multiplier $X$ in \eqref{multiplier}, we have
\begin{equation}
\begin{aligned}\label{right}
&\int\int\int|X\phi||F|r^2drd\omega dt\\
\leq& \int\int\int(t+2+r)^s\left|\phi_t+\phi_r+\frac{\phi}{r}\right||F|r^2drd\omega dt\\
&+\int\int\int(t+2-r)^s\left|\phi_t-\phi_r-\frac{\phi}{r}\right||F|r^2drd\omega dt\\
&\triangleq I_1+I_2.\\
\end{aligned}
\end{equation}
For $I_1$, by H\"{o}lder inequality we obtain for any $\delta>0$
\begin{equation}
\begin{aligned}\label{I1}
I_1=&\int\int\int(t+2+r)^s\left|\phi_t+\phi_r+\frac{\phi}{r}\right||F|r^2drd\omega du\\
\leq&\sup_{u}\left(\int\int_{t-r=u}(t+2+r)^s\left(\phi_t+\phi_r+\frac{\phi}{r}\right)^2r^2drd\omega\right)^{\frac12}\\
&\times\int\left(\int\int_{t-r=u}(t+2+r)^sF^2r^2drd\omega\right)^{\frac12}du\\
\leq&\left(\sup_{u}\int\int_{t-r=u}(t+2+r)^s\left(\phi_t+\phi_r+\frac{\phi}{r}\right)^2r^2drd\omega\right)^{\frac12}\\
&\times \left(\int\int\int(t+2-r)^{1+2\delta}(t+2+r)^sF^2r^2drd\omega du\right)^{\frac12}.
\end{aligned}
\end{equation}
In a similar way, we may estimate the term $I_2$ to get
\begin{equation}
\begin{aligned}\label{I2}
I_2
\leq&\left(\sup_{\underline{u}}\int\int_{t+r=\underline{u}}(t+2-r)^s\left(\phi_t-\phi_r-\frac{\phi}{r}\right)^2r^2drd\omega\right)^{\frac12}\\
&\times \left(\int\int\int(t+2+r)^{1+2\delta}(t+2-r)^sF^2r^2drd\omega d\underline{u}\right)^{\frac12}.
\end{aligned}
\end{equation}
We then conclude from \eqref{integrating}-\eqref{I2} that
\begin{equation}
\begin{aligned}\label{semifinal}
&\sup_{t}\int_{\R^3}\Big\{(t+2+r)^sr^2\left(\phi_t+\phi_r+\frac{\phi}{r}\right)^2+(t+2-r)^s\\
&\times\left(\phi_t-\phi_r-\frac{\phi}{r}\right)^2
+\left[(t+2+r)^s+(t+2-r)^s\right]\frac{(\nabla_{S^2}\phi)^2}{r^2}\Big\}dx\\
&\leq C_{\delta}\int\int\int(t+2-r)^{1+2\delta}(t+2+r)^sF^2r^2drd\omega dt\\
&+C_{\delta}\int\int\int(t+2+r)^{1+2\delta}(t+2-r)^sF^2r^2drd\omega dt.
\end{aligned}
\end{equation}
If we choose $\delta$ such that
\[
1+2\delta\leq s,
\]
then it follows from \eqref{semifinal} that
\begin{equation}
\begin{aligned}\label{final}
&\sup_{t}\int_{\R^3}\Big\{(t+2+r)^sr^2\left(\phi_t+\phi_r+\frac{\phi}{r}\right)^2+(t+2-r)^s\\
&\times\left(\phi_t-\phi_r-\frac{\phi}{r}\right)^2
+\left[(t+2+r)^s+(t+2-r)^s\right]\frac{(\nabla_{S^2}\phi)^2}{r^2}\Big\}dx\\
&\leq C_{\delta}\int\int\int(t+2-r)^{1+2\delta}(t+2+r)^sF^2r^2drd\omega dt.\\
\end{aligned}
\end{equation}
Direct computation yields that
\begin{equation}
\begin{aligned}\label{left}
&(t+2+r)^sr^2\left(\phi_t+\phi_r+\frac{\phi}{r}\right)^2+(t+2-r)^s\times\left(\phi_t-\phi_r-\frac{\phi}{r}\right)^2\\
=&\left[(t+2+r)^s+(t+2-r)^s\right]\left[\phi_t^2+\left(\phi_r+\frac{\phi}{r}\right)^2\right]\\
&+2\left[(t+2+r)^s-(t+2-r)^s\right]\phi_t\left(\phi_r+\frac{\phi}{r}\right)\\
=&2(t+2-r)^s\left[\phi_t^2+\left(\phi_r+\frac{\phi}{r}\right)^2\right]\\
&+\left[(t+2+r)^s-(t+2-r)^s\right]
\left(\phi_t+\phi_r+\frac{\phi}{r}\right)^2\\
\geq&2(t+2-r)^s\left[\phi_t^2+\left(\phi_r+\frac{\phi}{r}\right)^2\right].
\end{aligned}
\end{equation}
Then, the inequality \eqref{Weightedinequality} follows by combining \eqref{final}, \eqref{left} and the Hardy type inequalities showed in Lemma \ref{hardy} below. And we finish the proof of Theorem \ref{weighted}.
\end{proof}

\begin{lem}\label{hardy}
Let $\phi$ be the solution of problem \eqref{linearequation}. For $s>0$, we have
\begin{equation}
\begin{aligned}\label{hardy2}
\int_0^{\infty}(t+2-r)^s\phi^2dr\leq C\int_0^{\infty}(t+2-r)^s\left(\phi_r+\frac{\phi}{r}\right)^2r^2dr.
\end{aligned}
\end{equation}
and
\begin{equation}
\begin{aligned}\label{hardy1}
\int_0^{\infty}(t+2-r)^s\phi_r^2r^2dr\leq C\int_0^{\infty}(t+2-r)^s\left(\phi_r+\frac{\phi}{r}\right)^2r^2dr
\end{aligned}
\end{equation}
\end{lem}
Hereafter, $C$ denotes a positive constant which varies from line to line.
\begin{proof}
Obviously we have
\begin{equation}\label{1hardy}
\begin{aligned}
&\int_0^{\infty}(t+2-r)^s\phi_r^2r^2dr\\
\leq& 2\int_0^{\infty}(t+2-r)^s\left(\phi_r+\frac{\phi}{r}\right)^2r^2dr+2\int_0^{\infty}(t+2-r)^s\phi^2dr.
\end{aligned}
\end{equation}
By integration by parts, we get
\[
\begin{aligned}
&\int_0^{\infty}(t+2-r)^s\phi^2dr\\
=&-2\int_0^{\infty}r(t+2-r)^s\phi\phi_rdr+s\int_0^{\infty}r(t+2-r)^{s-1}\phi^2dr\\
=&-2\int_0^{\infty}r(t+2-r)^s\phi\left(\phi_r+\frac{\phi}{r}\right)dr+s\int_0^{\infty}r(t+2-r)^{s-1}\phi^2dr\\
&+2\int_0^{\infty}(t+2-r)^{s}\phi^2dr,\\
\end{aligned}
\]
which implies that
\[
\begin{aligned}
&\int_0^{\infty}(t+2-r)^{s}\phi^2dr\\
\leq&2\int_0^{\infty}r(t+2-r)^s\left|\phi\left(\phi_r+\frac{\phi}{r}\right)\right|dr\\
\leq& 2\left(\int_0^{\infty}(t+2-r)^{s}\phi^2dr\right)^{\frac12}\left(\int_0^{\infty}(t+2-r)^s
\left(\phi_r+\frac{\phi}{r}\right)^2dr\right)^{\frac12},\\
\end{aligned}
\]
where we used the fact that $s>0$. Then it follows that
\begin{equation}\label{2hardy}
\begin{aligned}
\int_0^{\infty}(t+2-r)^s\phi^2dr\leq C\int_0^{\infty}(t+2-r)^s\left(\phi_r+\frac{\phi}{r}\right)^2r^2dr,
\end{aligned}
\end{equation}
which is exactly \eqref{hardy2} in Lemma \ref{hardy}. The inequality \eqref{hardy1} follows from \eqref{1hardy} and \eqref{2hardy}.
\end{proof}

\section{Sobolev Embedding}
\par\quad
In this section, we will prove a weighted Sobolev embedding inequality for a function with compact support, by using Hardy type inequality, which is showed in the Lemma \ref{hardy3} below.
\begin{lem}\label{hardy3}
Suppose that a smooth function satisfies
\[
supp~\phi\subset \{|x|\le t+1\},
\]
then for $1<s<2$, we have
\begin{equation}\label{3hardy}
\begin{aligned}
&\left\||t+2-r|^{\frac s2-1}\phi\right\|_{L^2(\R^3)}\\
\leq& C\left(\left\||t+2-r|^{\frac s2}\phi_r\right\|_{L^2(\R^3)}+
\left\||t+2-r|^{\frac s2}\left(\frac{\phi}{r}\right)\right\|_{L^2(\R^3)}\right).
\end{aligned}
\end{equation}
\end{lem}
\begin{proof}
The proof of Lemma \ref{hardy3} can be showed by direct calculation, by using integration by parts and H\"{o}lder inequality.
\begin{equation}\nonumber
\begin{aligned}
&\int_{S^2}\int_0^{t+1}(t+2-r)^{s-2}\phi^2r^2dr d\omega\\
=&-\frac{1}{s-1}\int_{S^2}\int_0^{t+1}\phi^2r^2d(t+2-r)^{s-1}d\omega\\
\leq&C\int_{S^2}\int_0^{t+1}\left|(t+2-r)^{s-1}(2\phi\phi_rr^2+2\phi^2r)\right|drd\omega\\
\leq&C\left(\int_{S^2}\int_0^{t+1}(t+2-r)^{s-2}\phi^2r^2dr d\omega\right)^{\frac12}\\
&\times\Bigg\{\left(\int_{S^2}\int_0^{t+1}(t+2-r)^{s}\phi_r^2r^2dr d\omega\right)^{\frac12}\\
&+\left(\int_{S^2}\int_0^{t+1}(t+2-r)^{s}\left(\frac{\phi}{r}\right)^2r^2dr d\omega\right)^{\frac12}\Bigg\},
\end{aligned}
\end{equation}
and then \eqref{3hardy} follows.
\end{proof}

With this Hardy type inequality in hand, it is easy to get
\begin{thm}\label{sobolev}
Suppose that a smooth function satisfies
\[
supp~\phi\subset \{|x|\le t+1\},
\]
then for $1<s<2$, we have
\begin{equation}\label{sobolevinq}
\begin{aligned}
&\left\|r^{\frac12}(t+2+r)^{\frac12}(t+2-r)^{\frac{s-1}{2}}\phi\right\|_{L_r^{\infty}L^2(S^2)}\\
\leq&C\left(\left\|(t+2-r)^{\frac s2}\phi_r\right\|_{L^2(\R^3)}+\left\|(t+2-r)^{\frac s2}\left(\frac{\phi}{r}\right)\right\|_{L^2(\R^3)}\right).
\end{aligned}
\end{equation}
\end{thm}
\begin{proof}
We will divide the proof into two cases.\\
\textbf{Case 1: $t\geq 3r-2$.} In this case, we have
\[
t+2+r\leq 2(t+2-r).
\]
Then for \eqref{sobolevinq}, we only need to prove
\begin{equation}\label{sobolevinq1}
\begin{aligned}
&\left\|r^{\frac12}(t+2-r)^{\frac{s}{2}}\phi\right\|_{L_r^{\infty}L^2(S^2)}\\
\leq&C\left(\left\|(t+2-r)^{\frac s2}\phi_r\right\|_{L^2(\R^3)}+\left\|(t+2-r)^{\frac s2}\left(\frac{\phi}{r}\right)\right\|_{L^2(\R^3)}\right).
\end{aligned}
\end{equation}
Actually, by H\"{o}lder inequality, we have
\begin{equation}\label{ibp}
\begin{aligned}
&r(t+2-r)^{s}\int_{S^2}\phi^2(r\omega)d\omega\\
=&-r\int_{S^2}\int_r^{\infty}\partial_{\lambda}\left[(t+2-\lambda)^s\phi^2(\lambda\omega)\right]d\lambda d\omega\\
=&-2r\int_{S^2}\int_r^{\infty}(t+2-\lambda)^s\phi\phi_{\lambda}d\lambda d\omega\\
&+r\int_{S^2}\int_r^{\infty}s(t+2-\lambda)^{s-1}\phi^2d\lambda d\omega\\
\leq&C\left(\int_{\R^3}\frac{(t+2-r)^s}{r}\phi\phi_rdx+\int_{\R^3}\frac{(t+2-r)^{s-1}}{r}\phi^2dx\right)\\
\leq&C\left\|(t+2-r)^{\frac{s}{2}}\left(\frac{\phi}{r}\right)\right\|_{L^2(\R^3)}\\
&\times\left(\left\|(t+2-r)^{\frac{s}{2}}\phi_r\right\|_{L^2(\R^3)}+\left\|(t+2-r)^{\frac{s}{2}-1}\phi\right\|_{L^2(\R^3)}\right),\\
\end{aligned}
\end{equation}
and hence \eqref{sobolevinq1} follows from \eqref{3hardy} and \eqref{ibp}.\\
\textbf{Case 2: $t\leq 3r-2$.} In this case, we have
\[
t+2+r\leq 4r.
\]
Then for \eqref{sobolevinq}, it is sufficient to prove
\begin{equation}\label{sobolevinq2}
\begin{aligned}
&\left\|r(t+2-r)^{\frac{s-1}{2}}\phi\right\|_{L_r^{\infty}L^2(S^2)}\\
\leq&C\left(\left\|(t+2-r)^{\frac s2}\phi_r\right\|_{L^2(\R^3)}+\left\|(t+2-r)^{\frac s2}\left(\frac{\phi}{r}\right)\right\|_{L^2(\R^3)}\right).
\end{aligned}
\end{equation}
In a similar way as \eqref{ibp}, we have
\begin{equation}\label{ibp1}
\begin{aligned}
&r^2(t+2-r)^{s-1}\int_{S^2}\phi^2(r\omega)d\omega\\
=&-r^2\int_{S^2}\int_r^{\infty}\partial_{\lambda}\left[(t+2-\lambda)^{s-1}\phi^2(\lambda\omega)\right]d\lambda d\omega\\
=&-2r^2\int_{S^2}\int_r^{\infty}(t+2-\lambda)^{s-1}\phi\phi_{\lambda}d\lambda d\omega\\
&+r^2\int_{S^2}\int_r^{\infty}(s-1)(t+2-\lambda)^{s-2}\phi^2d\lambda d\omega\\
\leq&C\left\|(t+2-r)^{\frac{s}{2}}\phi_r\right\|_{L^2(\R^3)}\left\|(t+2-r)^{\frac{s}{2}-1}\phi\right\|_{L^2(\R^3)}\\
&+C\left\|(t+2-r)^{\frac{s}{2}-1}\phi\right\|_{L^2(\R^3)}^2.\\
\end{aligned}
\end{equation}
Again, \eqref{sobolevinq2} comes from \eqref{3hardy} and \eqref{ibp1}. Finally, we prove the desired inequality \eqref{sobolevinq}
by combining \eqref{sobolevinq1} and \eqref{sobolevinq2}.
\end{proof}
\section{SDGE for Semilinear Wave Equation in 3-D}
\par\quad
In this section, we are concerning the global existence for the following Cauchy problem for $1+\sqrt2<p<3$
\begin{equation}
\label{SDSemi3D}
\left\{
\begin{array}{l}
\d \partial_{t}^2\phi-\Delta \phi=|\phi|^p,
\quad (t, x)~ \mbox{in}\ ~\R_{+}\times \R^3,\\
\phi(0, x)=\e f(x),~~\phi_t(0, x)=\e g(x), \quad x\in\R^3,
\end{array}
\right.
\end{equation}
where $f, g\in C_0^{\infty}(\R^3)$ with compacted support
\[
supp~f, g\subset \{x: |x|\le 1\}.\\
\]

By combining the weighted $L^2-L^2$ estimate \eqref{Weightedinequality} in Theorem \ref{weighted} and inequality \eqref{sobolevinq} in Theorem \ref{sobolev}, and taking two spherical derivatives to the equation, we have for $1<s<2$
\begin{equation}\label{weighted2}
\begin{aligned}
&\left\|r^{\frac12}(t+2+r)^{\frac12}(t+2-r)^{\frac{s-1}{2}}\phi\right\|_{L_t^{\infty}L_r^{\infty}H^2(S^2)}\\
\leq&C\e+C_{\delta}\left\|(t+2+r)^{\frac s2}(t+r-2)^{\frac12+\delta}|\phi|^p\right\|_{L_t^2L_r^2H^2(S^2)}\\
\leq&C\e+C_{\delta}\left\|(t+2+r)^{\frac s2}(t+r-2)^{\frac12+\delta}\|\phi\|_{L^{\infty}(S^2)}^{p-1}
\|\phi\|_{H^2(S^2)}\right\|_{L_t^2L_r^2}.\\
\end{aligned}
\end{equation}
Here and afterwards, the norm $\|h(r)\|_{L_r^q}$ for $2\leq q<\infty$ denotes $\int \left|h(r)\right| r^2dr$. Taking $\frac{s-1}{2}=\frac 1p$, and by using the Sobolev embedding theorem on $S^2$
\[
\|\phi\|_{L^{\infty}(S^2)}\leq C\|\phi\|_{H^{2}(S^2)},
\]
it follows from \eqref{weighted2} that
\begin{equation}\label{weighted3}
\begin{aligned}
&\left\|r^{\frac12}(t+2+r)^{\frac12}(t+2-r)^{\frac{1}{p}}\phi\right\|_{L_t^{\infty}L_r^{\infty}H^2(S^2)}\\
\leq&C\e+C\left\|r^{\frac12}(t+2+r)^{\frac12}(t+2-r)^{\frac{1}{p}}\phi\right\|_{L_t^{\infty}L_r^{\infty}H^2(S^2)}^p\\
&\times \left\|(t+2+r)^{\frac12+\frac1p-\frac p2}(t+2-r)^{-\frac12+\delta}r^{-\frac p2}\right\|_{L_t^2L_r^2}.\\
\end{aligned}
\end{equation}
Let
\[
M\triangleq \left\|(t+2+r)^{\frac12+\frac1p-\frac p2}(t+2-r)^{-\frac12+\delta}r^{-\frac p2}\right\|_{L_t^2L_r^2},
\]
then direct computation implies that
\begin{equation}\label{M}
\begin{aligned}
M^2&=\int_0^{\infty}\int_0^{t+1}(t+2+r)^{1+\frac 2p-p}(t+2-r)^{-1+2\delta}r^{2-p}drdt\\
&\leq C\int_0^{\infty}(t+2)^{1+\frac2p-p}\left(\int_0^{t+1}(t+2-r)^{-1+2\delta}r^{2-p}dr\right)dt\\
&\leq C\int_0^{\infty}(t+2)^{1+\frac2p-p+2-p+2\delta}dt\\
&=C\int_0^{\infty}(t+2)^{-1+2\delta+\frac{2+4p-2p^2}{p}}dt.\\
\end{aligned}
\end{equation}
Since
\[
1+\sqrt2<p<3,
\]
which implies
\[
2+4p-2p^2<0,
\]
and hence we may choose $\delta$ small enough such that
\[
2\delta+\frac{2+4p-2p^2}{p}<0,
\]
which in turn yields that
\[
M\leq C.
\]
We conclude from \eqref{weighted3} that
\begin{equation}\label{iteration}
\begin{aligned}
&\left\|r^{\frac12}(t+2+r)^{\frac12}(t+2-r)^{\frac{1}{p}}\phi\right\|_{L_t^{\infty}L_r^{\infty}H^2(S^2)}\\
\leq&C\e+C\left\|r^{\frac12}(t+2+r)^{\frac12}(t+2-r)^{\frac{1}{p}}\phi\right\|_{L_t^{\infty}L_r^{\infty}H^2(S^2)}^p,\\
\end{aligned}
\end{equation}
then the global existence can be obtained by an iteration argument in a similar way as that in Lemma 1.3 in \cite{Georgive}.

\section{SDGE for Damped Semilinear Wave Equation in 3-D}
\par\quad
Finally, we consider the small data Cauchy problem for the semilinear wave equation with scale invariant damping in 3-D
\begin{equation}
\label{SDDSemi3D}
\left\{
\begin{array}{l}
\d \partial_{t}^2\Phi-\Delta \Phi+\frac{2}{1+t}\Phi_t=|\Phi|^p,
\quad (t, x)~ \mbox{in}\ ~\R_{+}\times \R^3,\\
\Phi(0, x)=\e f(x),~~\Phi_t(0, x)=\e g(x), \quad x\in\R^3,
\end{array}
\right.
\end{equation}
where $f, g\in C_0^{\infty}(\R^3)$ with compacted support
\[
supp~f, g\subset \{x: |x|\le 1\}.\\
\]
As mentioned in the introduction, by the transformation
\[
\phi(t, x):=(1+t)\Phi(t, x),
\]
problem \eqref{SDDSemi3D} can be rewritten as
\begin{equation}
\label{ChangedSDDSemi3D}
\left\{
\begin{array}{l}
\d \partial_{t}^2\phi-\Delta \phi=\frac{|\phi|^p}{(1+t)^{p-1}},
\quad (t, x)~ \mbox{in}\ ~\R_{+}\times \R^3,\\
\phi(0, x)=\e f(x),~~\phi_t(0, x)=\e \left\{f(x)+g(x)\right\}, \quad x\in\R^3.
\end{array}
\right.
\end{equation}

Next we will show that problem \eqref{ChangedSDDSemi3D}, or equivalently, \eqref{SDDSemi3D} has global solution for $p_c(5)<p<3$.
Since $p_c(5)<2$, we cannot take spherical derivatives twice to the equation due to the nonlinear term, hence in this case, we have to do more delicate analysis compared to the semilinear wave equation. On one hand, noting the Sobolev embedding
\[
H^1(S^2)\hookrightarrow L^q(S^2),~~~for~2\le q< \infty,
\]
which implies by combining the weighted $L^2-L^2$ estimate \eqref{Weightedinequality} in Theorem \ref{weighted}
\begin{equation}\label{Lq}
\begin{aligned}
&\left\|(t+2-r)^{\frac s2}\frac {\phi} {r}\right\|_{L_t^\infty L_r^2L^q(S^2)}\\
\le& \left\|(t+2-r)^{\frac s2}\frac \phi r\right\|_{L_t^\infty L_r^2L^2(S^2)}\\
&+\left\|(t+2-r)^{\frac s2}\frac {\nabla_{S^2}\phi} {r}\right\|_{L_t^\infty L_r^2L^2(S^2)}\\
\le& C \left\|(t+2+r)^{\frac s2}(t+2-r)^{\frac 12+\delta}F\right\|_{L^2(\R_+\times\R^3)}.
\end{aligned}
\end{equation}
On the other hand, by \eqref{sobolevinq} and \eqref{Weightedinequality}, we have
\begin{equation}\label{sobolevinq3}
\begin{aligned}
&\left\|r^{\frac12}(t+2+r)^{\frac12}(t+2-r)^{\frac{s-1}{2}}\phi\right\|_{L_t^\infty L_r^{\infty}L^2(S^2)}\\
\leq&C_\delta \left\|(t+2+r)^{\frac s2}(t+2-r)^{\frac 12+\delta}F\right\|_{L^2(\R_+\times\R^3)}.
\end{aligned}
\end{equation}
Then making interpolation between \eqref{Lq} and \eqref{sobolevinq3} with
\begin{equation}\label{intepower}
\begin{aligned}
\frac{1}{\sigma}=\frac{\theta}{2}+\frac{1-\theta}{\infty}, \frac{1}{\beta}=\frac{\theta}{q}+\frac{1-\theta}{2}
\end{aligned}
\end{equation}
and
\[
0<\theta\ll 1,
\]
we come to
\begin{equation}\label{sobolevinq5}
\begin{aligned}
&\left\|r^{\frac12-\frac{3\theta}{2}}(t+2+r)^{\frac{1-\theta}{2}}(t+2-r)^{\frac{s-1+\theta}{2}}\phi\right\|_{L_t^\infty L_r^{\sigma}L^{\beta}(S^2)}\\
\leq&C \left\|(t+2+r)^{\frac s2}(t+2-r)^{\frac 12+\delta}F\right\|_{L^2(\R_+\times\R^3)}.
\end{aligned}
\end{equation}

In order to prove global existence for problem \eqref{ChangedSDDSemi3D} with $p>p_c(5)$ by applying the above inequality, we also have to use the strong Huygen's principle for wave equations in 3-D. The solution for \eqref{linearequation} can be rewritten as
\begin{equation}\label{huygence}
\begin{aligned}
\phi(t, x)=C\int_0^t\int_{\R^3}\frac{\delta(t-\tau-|x-y|)}{t-\tau}F(\tau, y)dyd\tau,
\end{aligned}
\end{equation}
where $\delta(\cdot)$ denotes the Dirac delta function. On the characteristic cone$(t-\tau=|x-y|)$, it is easy to get
\[
t+2-|x|=\tau+|x-y|+2-|x|\leq 2+\tau+|y|,\\
\]
which yields for $\alpha>0$ by combining \eqref{huygence}
\begin{equation}\label{huygence1}
\begin{aligned}
(t+2-r)^{\alpha}\phi\leq C\int_0^t\int_{\R^3}\frac{\delta(t-\tau-|x-y|)}{t-\tau}(\tau+2+y)^{\alpha}|F(\tau, y)|dyd\tau.\\
\end{aligned}
\end{equation}
It holds that by combining \eqref{sobolevinq5}
\begin{equation}\label{weighted5}
\begin{aligned}
&\left\|r^{\frac12-\frac{3\theta}{2}}(t+2+r)^{\frac{1-\theta}{2}}(t+2-r)^{\frac{s-1+\theta}{2}+\alpha}\phi\right\|_{L_t^\infty L_r^{\sigma}L^{\beta}(S^2)}\\
\leq&C\Big\|r^{\frac12-\frac{3\theta}{2}}(t+2+r)^{\frac{1-\theta}{2}}(t+2-r)^{\frac{s-1+\theta}{2}}\\
&\times\int_0^t\int_{\R^3}
\frac{\delta(t-\tau-|x-y|)}{t-\tau}(\tau+2+y)^{\alpha}|F(\tau, y)|dyd\tau\Big\|_{L_t^\infty L_r^{\sigma}L^{\beta}(S^2)}\\
\leq&C \left\|(t+2+r)^{\frac s2+\alpha}(t+r-2)^{\frac12+\delta}|F|\right\|_{L^2(\R_+\times\R^3)}.\\
\end{aligned}
\end{equation}
In the same way as for the 3-D semilinear wave equation, applying \eqref{weighted5} to the Cauchy problem \eqref{ChangedSDDSemi3D}
after taking one spherical derivative,
we get
\begin{equation}\label{weighted6}
\begin{aligned}
&\left\|r^{\frac12-\frac{3\theta}{2}}(t+2+r)^{\frac{1-\theta}{2}}(t+2-r)^{\frac{s-1+\theta}{2}+\alpha}\phi\right\|
_{L_t^{\infty}L_r^{\sigma}W^{1, \beta}(S^2)}\\
\leq&C\e+C\left\|(t+2+r)^{\frac s2+\alpha}(t+2-r)^{\frac12+\delta}(1+t)^{-(p-1)}|\phi|^p\right\|_{L_t^{2}L_r^{2}H^1(S^2)}\\
\leq&C\e+C\left\|(t+2+r)^{\frac s2+\alpha-p+1}(t+2-r)^{\frac12+\delta}\|\phi\|_{L^{\infty}(S^2)}^{p-1}\|\phi\|_{H^1(S^2)}
\right\|_{L_t^{2}L_r^{2}}.\\
\end{aligned}
\end{equation}
Sine $q\geq 2$ and $0<\theta\ll 1$, we may get from \eqref{intepower} that $\beta>2$. Then Sobolev embedding
\[
W^{1, \beta}(S^2)\hookrightarrow L^{\infty}(S^2)
\]
and H\"{o}lder inequality yield that
\begin{equation}\label{weighted7}
\begin{aligned}
&\left\|r^{\frac12-\frac{3\theta}{2}}(t+2+r)^{\frac{1-\theta}{2}}(t+2-r)^{\frac{s-1+\theta}{2}+\alpha}\phi\right\|
_{L_t^{\infty}L_r^{\sigma}W^{1, \beta}(S^2)}\\
\leq&C\e+C\left\|(t+2+r)^{\frac s2+\alpha-p+1}(t+2-r)^{\frac12+\delta}\|\phi\|_{W^{1, \beta}(S^2)}^p
\right\|_{L_t^{2}L_r^{2}}\\
\leq& C\e+\left\|r^{\frac12-\frac{3\theta}{2}}(t+2+r)^{\frac{1-\theta}{2}}(t+2-r)^{\frac{s-1+\theta}{2}+\alpha}\phi\right\|
_{L_t^{\infty}L_r^{\sigma}W^{1, \beta}(S^2)}^p\\
&\times \left\|r^{-\left(\frac12-\frac{3\theta}{2}p\right)}(t+2+r)^{\frac{p(1-\theta)}{2}}(t+2-r)^{\frac12+\delta
-\left(\frac s2-\frac{1-\theta}{2}+\alpha\right)p}\right\|_{L_t^2L_r^{\frac{2\sigma}{\sigma-2p}}}.\\
\end{aligned}
\end{equation}
Let
\[
N\triangleq \left\|r^{-\left(\frac12-\frac{3\theta}{2}p\right)}(t+2+r)^{\frac{p(1-\theta)}{2}}(t+2-r)^{\frac12+\delta
-\left(\frac s2-\frac{1-\theta}{2}+\alpha\right)p}\right\|_{L_t^2L_r^{\frac{2\sigma}{\sigma-2p}}},
\]
and
\[
s=2-\frac1p,~\alpha=\frac{3}{2p}-\frac12,
\]
then direct computation yields that for $p<3$
\begin{equation}\label{N}
\begin{aligned}
N^2=&\int_0^{\infty}\Bigg(\int_0^{t+1}(t+2+r)^{\left(\frac32+\frac 1p-\frac{3p}{2}-\frac{p\theta}{2}\right)\frac{2}{1-p\theta}}\\
&\times(t+2-r)^{\left(-1+2\delta-\frac{p\theta}{2}\right)\frac{2}{1-p\theta}}
r^{2-\left(\frac12-\frac{3\theta}{2}\frac{2p}{1-p\theta}\right)}dr\Bigg)^{1-p\theta}dt\\
\leq& C\int_0^{\infty}(t+2)^{3+\frac2p-3p-p\theta}\left(\int_0^{t+1}(t+2-r)^{-1+\frac{2\delta-2p\theta}{1-p\theta}}r^{2-\frac{p-3p\theta}
{1-p\theta}}dr\right)^{1-p\theta}dt\\
\leq& C\int_0^{\infty}(t+2)^{3+\frac2p-3p+2-p+2\delta}dt\\
=&C\int_0^{\infty}(t+2)^{-1+2\delta-2p\theta+\frac{2+6p-4p^2}{p}}dt.\\
\end{aligned}
\end{equation}
Again, since
\[
p_c(5)<p<3,
\]
which implies
\[
2+6p-4p^2<0,
\]
and hence we may choose $\delta$ and $\theta$ small enough such that
\[
2\delta-2p\theta+\frac{2+6p-4p^2}{p}<0,
\]
which in turn yields that
\[
N\leq C.
\]
We conclude from \eqref{weighted6} that
\begin{equation}\label{iteration2}
\begin{aligned}
&\left\|r^{\frac12-\frac{3\theta}{2}}(t+2+r)^{\frac{1-\theta}{2}}(t+2-r)^{\frac{s-1+\theta}{2}+\alpha}\phi\right\|
_{L_t^{\infty}L_r^{\sigma}W^{1, \beta}(S^2)}\\
\leq& C\e+C\left\|r^{\frac12-\frac{3\theta}{2}}(t+2+r)^{\frac{1-\theta}{2}}(t+2-r)^{\frac{s-1+\theta}{2}+\alpha}\phi\right\|
_{L_t^{\infty}L_r^{\sigma}W^{1, \beta}(S^2)}^p,\\
\end{aligned}
\end{equation}
then the global existence can be obtained by an iteration argument in a similar way as that in Lemma 1.3 in \cite{Georgive}.

\begin{rem}
The restriction for $p<3$ comes from the fact $\alpha=\frac{3}{2p}-\frac12>0$.
\end{rem}

\section*{Acknowledgment}
\par\quad
The author would like to express his sincere thank to Prof. Yi Zhou for the helpful discussion and comments.

The author is supported by Natural Science Foundation of Zhejiang Province(LY18A010008), NSFC(11501273, 11726612, 11771359,
11771194), China Postdoctoral Science Foundation(2017M620128, 2018T110332), the Scientific Research Foundation of the First-Class Discipline of Zhejiang Province
(B)(201601).

\section*{References}
\bibliographystyle{plain}

\begin{thebibliography}{20}

\bibitem{D-L}
M.D'Abbicco and S.Lucente,
 NLWE with a special scale invariant
damping in odd space dimension,
\emph{Dynamical Systems, Differential Equations and Applications
AIMS Proceedings}, (2015), 312-319.

\bibitem{DLR15}
M.D'Abbicco, S.Lucente and M.Reissig,
A shift in the Strauss exponent for semilinear wave equations with a not effective damping,
\emph{J. Differential Equations}, \textbf{259} (2015), 5040-5073.


\bibitem{Georgive}
V. Georgiev, H. Lindblad and C. D. Sogge, Weighted Strichartz estimates and global existence
for semilinear wave equations, \emph{Amer. J. Math.}, \textbf{119} (1997), 1291-1319.
\bibitem{Glassey1}
R. T. Glassey, Finite-time blow-up for solutions of nonlinear wave equations, \emph{Math. Z.}, \textbf{177} (1981), 323-340.
\bibitem{Glassey2}
R. T. Glassey, Existence in the large for $\Box u=F(u)$ in two space dimensions, \emph{Math. Z.}, \textbf{178} (1981), 233-261.


\bibitem{IS}
M.Ikeda and M.Sobajima,
Life-span of solutions to semilinear wave equation
with time-dependent critical damping
for specially localized initial data,
\emph{Math. Ann.}, (2018),
https://doi.org/10.1007/s00208-018-1664-1.

\bibitem{IKTW}
T. Imai, M. Kato, H. Takamura and K. Wakasa, The lifespan of solutions of semilinear wave equations with the scale invariant
damping in two space dimensions, (in preparation).

\bibitem{John}
  F. John, Blow-up of solutions of nonlinear wave equations in three space dimensions,
\emph{Manuscripta Math.}, \textbf{28} (1979), 235-268.

\bibitem{KS}
M. Kato and M. Sakuraba, Global existence and blow-up for semilinear damped wave equations in three space dimensions,
(in preparation).
\bibitem{KTW}
M. Kato, H. Takamura and K. Wakasa, The lifespan of solutions of semilinear wave equations with the scale invariant damping
in one space dimension, (in preparation).

\bibitem{LTW}
N. A.Lai, H.Takamura and K. Wakasa,
Blow-up for semilinear wave equations
with the scale invariant damping and super-Fujita exponent,
\emph{J. Differential Equations}, \textbf{263(9)} (2017), 5377-5394.

\bibitem{LZ}
N. A. Lai, Y. Zhou, An elementary proof of Strauss conjecture, \emph{Journal of Functional Analysis}, \textbf{267} (2014), 1364-1381.


\bibitem{LS}H. Lindblad and C. D. Sogge, Long-time existence for small amplitude semilinear wave equations,
\emph{Amer. J. Math.}, \textbf{118}   (1996), 1047-1135.

\bibitem{Pa}A. Palmieri, A global existence result for a semilinear wave equation
with scale-invariant damping and mass in even space
dimension, arXiv: 1804. 03978v1.


\bibitem{Rammaha} M. A. Rammaha, Finite-time blow up for nonlinear wave equations in high dimensions, \emph{Comm. Partial Differential Equations}, \textbf{12} (1987), 677-700.

\bibitem{Schaeffer}
J. Schaeffer, The equation $\Box u=|u|^p$ for the critical value of $p$, \emph{Proc. Roy. Soc. Edinburgh
Sect. - A}, \textbf{101}  (1985), 31-44.

\bibitem{Sideris}T.C. Sideris, Nonexistence of global solutions to semilinear wave equations in high dimensions, \emph{J. Differential Equations}, \textbf{52} (1984), 378-406.


\bibitem{Strauss}
  W. A. Strauss, Nonlinear scattering theory at low energy, \emph{J. Funct. Anal.}, \textbf{41} (1981), 110-133.

  \bibitem{Takamura}
  H. Takamura and K. Wakasa, The sharp upper bound of the lifespan of solutions to critical semilinear wave equations in high
dimensions, \emph{J. Differential Equations}, \textbf{251} (2011), 1157-1171.

\bibitem{Tu}
Z.Tu and J.Lin,
A note on the blowup of scale invariant damping wave equation
with sub-Strauss exponent,
arXiv:1709.00866v2.


\bibitem{wak16}
K.Wakasa, The lifespan of solutions to semilinear damped wave equations in one space dimension,
\emph{Communications on Pure and Applied Analysis}, \textbf{15} (2016), 1265-1283.

\bibitem{Yordanov}
B. Yordanov and Q. S. Zhang, Finite time blow up for critical wave equations in high dimensions,
\emph{J. Funct. Anal.}, \textbf{231} (2006), 361-374.

\bibitem{Zhou1} Y. Zhou, Cauchy problem for semilinear wave equations with small data in four space dimensions, \emph{J. Partial Differential Equations}, \textbf{8} (1995), 135-144.

\bibitem{Zhou2}
Y. Zhou, Blow up of solutions to semilinear wave equations with critical exponent in high
dimensions, \emph{Chinese Ann. Math. Ser. - B}, \textbf{28} (2007), 205-212.


\end{thebibliography}

\end{document}